\documentclass[12pt,twoside,english]{article}
\usepackage{amsfonts,babel,amssymb,amsmath}
\usepackage{graphicx}

\newtheorem{proposition}{Proposition}[section]
\newtheorem{corollary}[proposition]{Corollary}
\newtheorem{lemma}[proposition]{Lemma}
\newtheorem{theorem}[proposition]{Theorem}

\newenvironment{proof}{\noindent{\em Proof}. }{\hfill$\Box$\vspace{3pt}}

\title{Hilbert scales and Sobolev spaces defined by associated Legendre functions}

\author{V\'{\i}ctor Dom\'{\i}nguez\footnote{Dep. Ingenier\'{\i}a Matem\'{a}tica e Inform\'{a}tica,
Universidad P\'{u}blica de Navarra, Campus de Tudela, 31500 Tudela,
Spain. E--mail: {\tt  victor.dominguez@unavarra.es}. Research
partially supported by Spanish MEC Project MTM2007--63204}\and Norbert
Heuer\footnote{Facultad de Matem\'aticas, Pontificia Universidad
Cat\'olica de Chile, Casilla 306, Correo 22, Santiago, Chile.
E--mail: {\tt nheuer@mat.puc.cl} This author is supported by
Fondecyt-Chile under grant no. 1080044.} \and  Francisco--Javier
Sayas\footnote{Dep. Matem\'{a}tica Aplicada, CPS, Universidad de
Zaragoza, 50018 Zaragoza, Spain \& School of Mathematics, University
of Minnesota, Minneapolis, MN 55455, USA. E--mail: {\tt
sayas002@umn.edu}. Research partially supported by Spanish MEC
Project MTM2007--63204 \& Gobierno de Arag\'{o}n (Grupo Consolidado
PDIE)}}

\date{}

\begin{document}

\maketitle

\begin{abstract}
In this paper we study the Hilbert scales defined by the associated
Legendre functions for arbitrary integer values of the parameter.
This problem is equivalent to study the left--definite spectral
theory associated to the modified Legendre equation. We give several
characterizations of the spaces as weighted Sobolev spaces and prove
identities among the spaces corresponding to lower regularity index.
\end{abstract}

\section{Introduction and motivation}

In this paper we deal with the Hilbert scales generated by the sets
of Legendre functions. The concept of Hilbert scale is related (in
many ways it is equivalent and leads to the same sets) to the left--
and right--definite spectral theories for differential operators and
to the construction of the so--called Sobolev towers. We will
clarify this point in the following section. We first remark that
the study of the Hilbert scales defined by Legendre polynomials (and
actually by some other sets of orthogonal polynomials) has been
subject
 of active research (see \cite{ArLiMa:2002},
\cite{BrLiTuWe:2009}, \cite{EvLiWe:2002}). The aim is often the
study of the spaces of the domain of definition of iterated powers
(and also powers of the square root) of a given unbounded
self--adjoint operator that stems from a differential equation,
whose spectral set is a known sequence of orthogonal polynomials.
Part of the interest of that study is being able to characterize the
spaces as weighted Sobolev spaces.

In this paper we extend the study of the Hilbert scales for Legendre
polynomials to the sequence of Hilbert scales defined by associated
Legendre functions, which are the eigenfunctions of the modified
Legendre operator in $(-1,1)$
\[
-((1-t^2) u')' + \frac{m^2}{(1-t^2)}\,u
\]
for positive integer values of $m$ (the case $m=0$ corresponds to
the Legendre polynomials). Apart from the theoretical interest of
extending this study to new families of spaces, including some new
results where we will be able to describe the spaces in several
different ways and to identify the `central part' of the Hilbert
scales for different values of $m$, we are now going to try to
motivate this study from the point of view of ongoing research in
boundary integral operators of sphere--like bodies.

A basis of spherical harmonics can be built as follows. Consider the
normalized associated Legendre functions $Q^m_n$ (the precise
definition is given in \eqref{Legendrefunctions} and \eqref{defQnm}
in terms of the Legendre polynomials). Then, the functions
\[
Y^m_n(\theta,\varphi):= Q^{|m|}_n(\cos\theta) \exp(\imath\, m
\,\varphi), \qquad n=0,1,\ldots, \qquad -n\le m\le n
\]
form an orthonormal basis of $L^2(S)$, $S$ being the three
dimensional sphere parametrized in spherical coordinates,
$(\theta,\varphi)\in [0,\pi]\times[0,2\pi]$, which has $\sin\theta
\mathrm d\theta\,\mathrm d\varphi$ as surface element. The Sobolev
spaces on the sphere $H^s(S)$ for $s>0$ can be described as the set
of functions $g$ for which
\[
\sum_{n=0}^\infty \sum_{m=-n}^n (2n+1)^{2s} |\widehat g_{n,m}|^2 <
\infty,
\]
where $\widehat g_{n,m}$ are the Fourier coefficients of $g$ with
respect to the Hilbert basis $Y^m_n$ of $L^2(S)$. For negative
values of $s$, a norm can be defined by duality (pivoting around
$L^2(S)$) or by completing the space of spherical harmonic
polynomials (linear combinations of the $Y^m_n$) with respect to the
norms above, which are well defined when the number of non--zero
coefficients is finite. Most of the work related to integral
operators (which is very relevant in scattering theory:
\cite{ColtonKress1}, \cite{ColtonKress2}, \cite{Nedelec}) uses the
basis of spherical harmonics by grouping in terms of the
eigenvalues. This means that we group the functions $Y^m_n$ for
$-n\le m\le n$ and then consider all values of $n$. This treatment
gives an orthogonal decomposition of the spaces $H^s(S)$ as a sum of
finite dimensional spaces (with dimensions growing linearly in $n$).
However, we can think of the spaces defined by closuring
$\mathrm{span}\,\{ Y^m_n \,:\, n \ge |m|\}$ with the Sobolev norms
above. This means that we group the terms in the norm differently
like
\[
\sum_{m=-\infty}^\infty \Big( \sum_{n=|m|}^\infty (2n+1)^{2s}
|\widehat g_{n,m}|^2\Big).
\]
Apart from the exponential common factor $\exp(\pm \imath m
\varphi)$ in $Y^m_n$, what we have to study are then spaces created
by closuring $\mathrm{span}\,\{ Q^m_n(\cos\theta)\,:\, n \ge m\}$
for different values of $m\ge 0$. This will give a different
orthogonal decomposition of the Sobolev spaces on the sphere as a
countable sum of infinite--dimensional spaces, that are basically
spaces defined along the generatrix of the sphere (a half--circle)
rotated and multiplied with the functions $\exp(\pm \imath
m\varphi)$. Finally, instead of working with the functions
$Q^m_n(\cos\theta)$ and including the $\sin\theta$ weight from the
surface measure, we can directly study the spaces related to
$\mathrm{span}\,\{ Q^m_n\,:\, n\ge m \}$ as subspaces of $L^2(-1,1)$
for different values of $m$. This study is motivated by our wish to
understand the behavior of the sequence of one--dimensional integral
equations that arises when some particular numerical methods are
applied to boundary integral equations of the sphere or on any
smooth axisymmetric body in $\mathbb R^3$
\cite{DominguezHeuerSayas}.

The paper is structured as follows. In Section \ref{app:sec1} we
introduce the Hilbert scales defined by the associated Legendre
functions after having shown how Hilbert scales can be defined in
several equivalent ways. The spaces will be denoted $H_m^s$ where
$m\ge 0$ is the parameter in the Legendre function and $s\in \mathbb
R$ is the regularity index. The case $m=0$ will correspond to the
well--studied case of the left--definite spectral theory of
Legendre's equation. In Section \ref{section:Y}, we show that the
abstract construction of the Hilbert scales departing from the
Legendre function is equivalent to the constructions that arise from
both the weak and strong forms of the modified Legendre differential
equation. In particular, we show how the spaces $H_m^1$ coincide for
all values of $m\neq 0$. We advance in the same direction in Section
\ref{section:Z} to prove that all the spaces $H_m^2$ for $m\ge 2$
are equal and that $H_0^2$ and $H_1^2$ are strict different
supersets of them. In Section \ref{oldsection4} we give an
alternative expression of all the spaces $H_m^k$ for non--negative
integer value of $k$ as weighted Sobolev spaces. We use this
characterization in Section \ref{sec:further} to prove some useful
additional properties of these spaces. Finally in Section
\ref{sec:lastsection}, we show that given $k\ge 0$ all the spaces
$H_m^k$ are equal for $m\ge k$ and study how the remaining ones
behave. A first appendix is devoted to collect several purely
technical lemmas and a second one to sketch how the particular case
$m=0$ (which had already been analyzed in the literature) can be
studied with the techniques of this paper.

\paragraph{Background material.} Throughout the paper we will be
using the space
\[
\mathcal D (-1,1) = \{ \varphi \in \mathcal C^\infty(-1,1)\,:\,
\mathrm{supp}\,\varphi \subset (-1,1)\}
\]
and the classical Sobolev spaces $H^k(-1,1)$ and $H^k_0(-1,1)$ for
non--negative integer values of $k$. For elementary properties of
these spaces we refer the reader to any textbook or monograph on
Sobolev theory or elliptic PDEs (for example, \cite{Adams}). We will
also use the spaces
\begin{eqnarray*}
H^k_{\mathrm{loc}}(-1,1)&=& \{ u : (-1,1) \to \mathbb R\,:\,
u|_{(-a,a)} \in H^k(-a,a),\quad \forall a\in (0,1)\}\\
&=& \{ u : (-1,1) \to \mathbb R\,:\, \varphi\,u \in H^k(-1,1), \quad
\forall \varphi \in \mathcal D(-1,1)\}.
\end{eqnarray*}
The $L^1-$based Sobolev space,
\[
W^{1,1}(-1,1):= \{ u\in L^1(-1,1)\,:\, u'\in L^1(-1,1)\}
\]
will also appear in the sequel. All along the paper, derivatives
will be understood in the sense of distributions in $(-1,1)$.
Whenever a derivative of a function appears, it will be implicit
that from the conditions given to the function, it can be proved
that the function is locally integrable in $(-1,1)$ and therefore it
can be understood as a distribution. The measure in all integrals
will be the Lebesgue measure and we will commonly shorten
\[
\int_{-1}^1 f = \int_{-1}^1 f(t) \mathrm d t
\]
to alleviate many expressions to come from the explicit presence of
both the variable and the symbol for the Lebesgue measure.

We will also make repeated use of this elementary form of the
integration by parts formula.

\begin{lemma}\label{intbyparts}
Assume that $g \in W^{1,1}_0(-1,1)=\{u \in W^{1,1}(-1,1) \,:\,
u(-1)=u(1)=0\} $ and that $g'=g_1+g_2$ with $g_1, g_2 \in
L^1(-1,1)$; then
\[
\int_{-1}^1 g_1 =-\int_{-1}^1 g_2.
\]
\end{lemma}

\section{Definitions}\label{app:sec1}

\paragraph{First construction of Hilbert scales.} The following
construction is based on how Hilbert scales are commonly introduced
in the literature of integral equations: see \cite{PrSi:1991}, which
bases this part in \cite{Be:1968}. Let $H^0$ be a separable real
Hilbert space and $\{ \psi_n\}$ an orthonormal basis of the space.
We consider a sequence of positive numbers $\lambda_n$ such that
$\lambda_n \to \infty$. The following collection of norms for $s\in
\mathbb R$
\[
\| u\|_s := \left( \sum_{n=1}^\infty
\lambda_n^{2s}|(u,\psi_n)_{H^0}|^2\right)^{1/2},
\]
is well defined in the set $\mathbb T:=\mathrm{span}\,\{\psi_n\,:\,
n\ge 1\}$. For $s>0$ we can define
\[
H^s:=\{ u\in H^0\,:\, \| u\|_s < \infty\}.
\]
For negative $s$ we have two options: (a) take the completion of
$\mathbb T$ with the norm $\|\hspace{2pt}\cdot\hspace{2pt}\|_s$; (b)
define $H^s$ as the representation of the dual space of $H^{-s}$
when $H^0$ is identified with its dual space. Both constructions
lead to isometrically isomorphic definitions of the spaces $H^s$ for
negative $s$. The resulting chain of spaces is what is known as a
Hilbert scale (see \cite{PrSi:1991} and \cite{Be:1968}). Note that
$H^r\subset H^s$ for all $r>s$ with compact and dense inclusion. We
also have the direct estimate for the size of the norms
\[
\| u\|_{s} \le (\min_n \lambda_n)^{-\varepsilon}
\|u\|_{s+\varepsilon}, \qquad \forall u \in H^{s+\varepsilon},
\qquad \forall \varepsilon >0.
\]
The set $\mathbb T$ is dense in $H^s$ for all $s$. An element of
$H^s$ can be written as a convergent series
\[
u= \sum_{n=1}^\infty (u,\psi_n)_{H^0}\,\psi_n \qquad
\mbox{(convergent in $H^s$)},
\]
where the coefficients $(u,\psi_n)_{H^0}$ are defined as $H^0$ inner
products when $s>0$ and as duality products for $s<0$. Moreover the
couple formed by $H^s$ and $H^{-s}$ is a dual pair, where each of
the spaces can be understood as the dual space of the other one and
its duality bracket is just the extension of the $H^0$ inner
product. The spaces $H^s$ are interpolation spaces, so $[ H^r ,
H^s]_\gamma = H^{(1-\gamma)\,r+\gamma\,s}$ for any $r\neq s$ and
$\gamma \in (0,1)$. \hfill $\Box$\vspace{3pt}

\paragraph{Second construction.} Assume now that $X$ is another
Hilbert space such that $X \subset H^0$ with compact and dense
inclusion. Then, we can define the operator $G: H^0 \to H^0$ that
associates $u=Gf$, where $u$ solves
\[
u \in X, \qquad \mbox{s.t.}\qquad (u,v)_X=(f,v)_{H^0}, \quad \forall
v\in X.
\]
Then $G$ is selfadjoint, compact and positive definite. Therefore by
Hilbert--Schmidt's theorem, we can write
\[
Gf= \sum_{n=1}^\infty \lambda_n^{-2} (f,\phi_n)_{H^0}\, \phi_n,
\]
where: (a) $\{ \phi_n\}$ is an orthonormal basis of $H^0$; (b) the
sequence $\{ \lambda_n^2\}$ is non--decreasing and diverges to
$+\infty$; (c) $\{ \phi_n\}$ is complete orthogonal in $X$.

Note that $\|\phi_n\|_X=\lambda_n$ and that $(\phi_n;\lambda_n^2)$
are eigenpairs for the problem that defines $G$:
\[
(\phi_n,v)_X= \lambda_n^2 (\phi_n,v)_{H^0} \qquad \forall v\in X.
\]
The sequence $\{ (\phi_n;\lambda_n)\,:\, n\ge 1\}$ defines a Hilbert
scale $H^s$ and we can see that: $H^1=X$ with the same inner
product; the space $H^s$ is the range of the $s/2-$th power of $G$
\[
H^s=\mathcal R(G^{s/2}), \qquad \forall s>0.
\]
The operator $G$ can be naturally extended (restricted when $s>0$)
to $G:H^s \to H^{s+2}$ and it is an isometric isomorphism between
these pairs of spaces. Furthermore $H^{-1}$ is the representation of
$X'$ that appears when we identify $(H^0)'$ with $H^0$, i.e.,
\[
X=H^1 \subset H^0 \cong (H^0)'\subset H^{-1}=X'
\]
is a Gelfand triple. We can restart the construction from the point
of view of the unbounded selfadjoint operator $A:D(A) \subset H^0
\to H^0$, where $D(A)=\mathcal R(G)$ and $A=G^{-1}$ in this domain.
From the point of view of $A$, $(\phi_n; \lambda_n^2)$ are just the
eigenpairs of the associated Sturm--Liouville problem. The spaces
for integer values constitute a Sobolev tower in the sense described
in \cite[Chapter 2]{EngelNagel}. If we start with a symmetric
differential operator $A$, different choices of $D(A)$ will lead to
different Hilbert scales: for instance, $Au:=-u''+u$ leads to the
Fourier series of sines (by taking Dirichlet conditions), cosines
(Neumann conditions) or sines and cosines (periodic conditions).
Starting the construction directly from the differential operator,
an adequate choice of boundary conditions leads to the concept of
left--definite spectral theory, which is equivalent to this
construction.

\paragraph{Hilbert scales with Legendre functions.}
Consider the Legendre polynomials
\[
P_n(t):= \frac{(-1)^n}{2^n n!} \frac{\mathrm d^n}{\mathrm d t^n}
\big( (1-t^2)^n\big), \qquad n\ge 0
\]
and the associated Legendre functions
\begin{equation}\label{Legendrefunctions}
P_n^m(t):= (1-t^2)^{m/2} P_n^{(m)}(t), \qquad 0 \le m \le n,
\end{equation}
normalized as
\begin{equation}\label{defQnm}
Q_n^m:=c_{n,m} P_n^m , \qquad c_{n,m}:=\left( \frac{2n+1}2\,
\frac{(n-m)!}{(n+m)!}\right)^{1/2}, \qquad 0\le m\le n.
\end{equation}
Then, for any $m$, $\{ Q_n^m \,:\, n\ge m\}$ is an orthonormal
basis of $L^2(-1,1)$. From now on we will write
\[
\omega(t):= \sqrt{1-t^2}.
\]
The modified Legendre differential operator
\[
\mathcal L_m  u:= -(\omega^2 u')'+m^2 \omega^{-2}u
\]
has $Q_n^m$ as eigenfunctions:
\begin{equation}\label{Qnm}
 \mathcal L_m Q_n^m = n(n+1)\, Q_n^m.
\end{equation}

{\em Let $H_m^s$ with $s\in \mathbb R$ be the Hilbert scale defined
by $\{ (Q_n^m; 2n+1)\,:\, n\ge m\}$. Their respective norms will be
denoted $\|\hspace{2pt}\cdot\hspace{2pt}\|_{m,s}$.} The upper index
in all spaces will be a regularity index, taking values in all of
$\mathbb R$, while we use the lower index (the integer $m\ge 0$) to
mark the different scales. As a reminder, for $\mathbb R \ni s\ge 0$
and $\mathbb Z \ni m\ge 0$, elements of $H_m^s$ are functions
$u:(-1,1) \to \mathbb R$ such that
\[
\| u\|_{m,s}^2 = \sum_{n=m}^\infty (2n+1)^{2s} |
(u,Q_n^m)_{L^2(-1,1)}|^2< \infty.
\]
Since
\[
2 \le \frac{2n+1}{\sqrt{n(n+1)}} \le \frac3{\sqrt2},\qquad n\ge 1,
\]
the scales can be defined equivalently with the pairs $(Q_n^m;
\sqrt{n(n+1)})$, with the only exception of the scale associated to
$m=0$, where we have to take $(Q_0^0; 1)$ instead of $(Q_0^0;0)$ to
avoid cancelation of the first coefficient. {\em We will keep the
first choice to avoid this singular case and also to fit into the
frame of spherical harmonics which is the original motivation of
this work.} Modifications to use $\sqrt{n(n+1)}$ or the even simpler
values $n+1$ are simple, although they change the values of the
constants in many of the inequalities to follow.

The space of univariate polynomials will be denoted by $\mathbb P$,
with $\mathbb P_n$ denoting the space of polynomials of degree not
greater than $n$. A relevant set throughout will be
\[
\omega^m \mathbb P = \{\omega^m \,p\,:\,p\in \mathbb P\} =
\mathrm{span}\{ Q_n^m \,:\, n\ge m\}=\mathrm{span}\{ P_n^m \,: n \ge
m\}.
\]
For easy reference, let us write down  two properties of the
weight function $\omega$:
\begin{equation}\label{app:omega1}
(\omega^\beta)'=-\beta\, t\, \omega^{\beta-2},
\end{equation}
\begin{equation}\label{app:omega2}
\omega^\alpha \in L^2(-1,1) \qquad \Longleftrightarrow \qquad \alpha
> -1.
\end{equation}
In \eqref{app:omega1} we have denoted by $t$ the monomial of degree
one, that is, the function $p(t)=t$. We will maintain this notation
henceforth.

\section{An alternative definition}\label{section:Y}

The study of the spaces generated by Legendre polynomials, which
corresponds to $m=0$ in the present work,  has already been
undertaken in \cite{EvLiWe:2002} (see also \cite{ArLiMa:2002} for a
previous study and \cite{BrLiTuWe:2009} for a more general theory
covering some classical families of orthogonal polynomials). The
analysis there is based on rewriting  the inner product defined by
the powers of the Legendre differential operator. Roughly speaking,
by integrating by parts, this product is shown to be equal to a sum
of weighted $L^2$ inner products of  the derivatives of the
functions. Hence, as a simple byproduct, these spaces are identified
with weighted Sobolev spaces, namely,
\begin{equation}\label{Littlejohn}
H^k_0 = \{ u\in L^2(-1,1)\,:\, \omega^\ell u^{(\ell)} \in L^2(-1,1),
\quad 0 \le \ell \le k\}
\end{equation}
with equivalent norms. Note that this is going to be a particular
case of the study we do in Section \ref{oldsection4} (see Theorem
\ref{app:the2}), where we generalize the above result for any $m$.

In this section, we are going to derive again the Hilbert scales
$H_m^s$ for $m\ge 1$ using the weak form of Legendre's equation.
This will serve us to start obtaining weighted Sobolev type
expressions for the spaces for positive integer values of the
regularity index and to conclude properties on how the scales
coincide for small values of the regularity index. The corresponding
theory for Legendre polynomials (the case $m=0$ in this paper) is
already known (see \cite{ArLiMa:2002}). For the sake of
completeness, we will sketch the basic results (in parallel to those
of this section for $m\ge 1$) in Appendix B.

Let
\[
Y:= \{ u:(-1,1) \to \mathbb R \,:\, \omega^{-1}u, \quad \omega u'\in
L^2(-1,1)\},
\]
endowed with the norm
\[
\left(\int_{-1}^1 \omega^2 |u'|^2 +\int_{-1}^1 \omega^{-2}
|u|^2\right)^{1/2}.
\]
It is simple to see that $Y$ is a Hilbert space. Since $u= \omega
(\omega^{-1}u)$ and $|\omega(t)|\le 1$, then $Y\subset L^2(-1,1)$.

\begin{proposition}\label{prop:YH1loc}
$Y \subset H^1_{\mathrm{loc}}(-1,1) \subset \mathcal C(-1,1).$
\end{proposition}

\begin{proof}
Let $u\in Y$. Using the fact that $\omega^{-1}\in \mathcal
C^\infty(-1,1)$, it follows that $u'=\omega^{-1}(\omega u')\in
L^2_{\mathrm{loc}}(-1,1)$ and therefore $u\in H^1(-a,a)\subset
\mathcal C[-a,a]$ for all $0<a<1$.
\end{proof}

Since the rule $(u^2)'=2u\,u'$ holds in $H^1_{\mathrm{loc}}(-1,1)$,
it holds in $Y$. Note also that $\mathbb P_0 \cap Y =\{ 0\}$. The
following set
\[
\mathcal C_0:= \{u\in \mathcal C[-1,1]\,:\, u(-1)=u(1)=0\}
\]
will be relevant in the sequel.

\begin{proposition}\label{app:propY1} $Y \subset \mathcal C_0$ with
continuous injection.
\end{proposition}

\begin{proof} Using the definition we prove that if $u\in Y$,
then $u\,u'=(\omega^{-1}u)(\omega u')\in L^1(-1,1)$. Therefore $u^2
\in W^{1,1}(-1,1) \subset \mathcal C[-1,1]$ and $u \in \mathcal
C[-1,1]$. Since $\omega^{-1}u\in L^2(-1,1)$ and $u$ is continuous
near the two  singularities of $\omega^{-1}$,
necessarily $u$ has to vanish in both of them.
\end{proof}

\begin{lemma}\label{lemma:H10} If $u \in H^1_0(-1,1)$ then:
\begin{itemize}
\item[{\rm (a)}] $\omega^{-1}u \in \mathcal C_0$,
\item[{\rm (b)}] $\omega\,u \in H^1_0(-1,1)$.
\end{itemize}
\end{lemma}

\begin{proof} Let $u \in H^1_0(-1,1)$. Note first that $\omega^{-1}u \in \mathcal
C(-1,1)$. We can write
\[
u(t)=\int_{-1}^t u'(s)\mathrm d s
\]
and therefore
\[
|u(t)| \le \sqrt{1+t}\Big( \int_{-1}^t |u'(s)|^2 \mathrm d
s\Big)^{1/2}, \qquad \forall t \in [-1,1].
\]
Using Lebesgue's Theorem we  prove that
\[
|\omega^{-1}(t)u(t)| \le \frac1{\sqrt{1-t}} \Big( \int_{-1}^t
|u'(s)|^2 \mathrm d s\Big)^{1/2} \stackrel{t\to -1}{\longrightarrow}
0.
\]
The limit for $t \to 1$ is obtained similarly.

To prove the second statement, note first that $\omega \,u\in
L^2(-1,1)$ (because $\omega$ is bounded) and
\[
(\omega\,u)'=- t \,\omega^{-1}u+ \omega\,u' \in L^2(-1,1)
\]
because of the first part of the Lemma. Finally, $\omega\,u\in
\mathcal C[-1,1]$ and the limits at both extreme points of the
interval are zero because both $\omega$ and $u$ vanish there.
\end{proof}

\begin{proposition}\label{app:propY2} For all $m\ge 0$
\[
\{\omega^m \,u\,:\, u\in H^1_0(-1,1)\}\subset Y.
\]
\end{proposition}

\begin{proof} From Lemma \ref{lemma:H10}(a) and the boundedness of
$\omega$ it is clear that $H^1_0(-1,1) \subset Y$. By Lemma
\ref{lemma:H10}(b) it follows that if $u \in H^1_0(-1,1)$, then
$\omega^m\, u \in H^1_0(-1,1)$ for $m\ge 1$, which finishes the
proof.
\end{proof}

\begin{lemma}\label{app:reg3} If $u\in Y$, $p$ is a polynomial and $m\ge 1$, then
\[
\omega^2 (\omega^m p)'u \in H^1_0(-1,1).
\]
\end{lemma}

\begin{proof}
Assume first that {\em $m$ is even}. Then $\omega^m p\in \mathbb P$
and $\omega^2(\omega^m p)'\in \mathbb P$, so $\omega^2 (\omega^m
p)'u\in \mathcal C[-1,1]\subset L^2(-1,1)$. Taking the derivative
\[
(u \omega^2 (\omega^m p)')'=u' \omega^2 (\omega^m p)'+u(\omega^2
(\omega^m p)')'.
\]
Note $\omega u'\in L^2(-1,1)$, $\omega$ is bounded and $(\omega^m
p)'$ is a polynomial, as is $(\omega^2 (\omega^m p)')'$. This proves
that $(u \omega^2 (\omega^m p)')'\in L^2(-1,1)$.

If {\em $m$ is odd}, then $\omega(\omega^m p)'\in \mathbb P$ and
therefore $\omega^2 (\omega^m p)'u\in \mathcal C[-1,1]\subset
L^2(-1,1)$. For the derivative we decompose
\[
(u \omega^2 (\omega^m p)')' = (u'\omega-t u\omega^{-1}) \omega
(\omega^m p)' +u\omega (\omega (\omega^m p)')'
\]
and note that
\[
u' \omega, \,u\omega^{-1},\, u\omega \in L^2(-1,1), \qquad \omega
(\omega^m p)'\in \mathbb P,
\]
so the derivative is in $L^2(-1,1)$.

Finally, in both cases the function is continuous and vanishes at
the extremes of the interval. Therefore it is in $H^1_0(-1,1)$.
\end{proof}

\begin{proposition}\label{app:dens3} $\omega^m \mathbb P$ is dense in $Y$ for all
$m\ge 1$.
\end{proposition}

\begin{proof}
Suppose that
\[
\int_{-1}^1 \omega^2 u' (\omega^m p)' + m^2 \int_{-1}^1 \omega^{-2}
u\,\omega^m \, p =0, \qquad \forall p\in \mathbb P.
\]
By Lemma \ref{app:reg3} we can apply integration by parts and obtain
\[
\int_{-1}^1 u\,\left( -(\omega^2 (\omega^mp)')'+\frac{m^2}{\omega^2}
\omega^m p\right)=\int_{-1}^1 u \,\mathcal L_m (\omega^m p)=0,
\qquad \forall p\in \mathbb P.
\]
Taking $p=P_n^{(m)}$ or equivalently $\omega^m p=P_n^m$ and applying
that
\[
\mathcal L_m P_n^m = n(n+1) P_n^m ,
\]
it follows that
\[
n(n+1) \int_{-1}^1 u \, P_n^m =0, \qquad \forall m\ge 1.
\]
This implies that $u=0$ since $\{ P_n^m\,:\, n\ge m\}$ is an
orthogonal basis of $L^2(-1,1)$.
\end{proof}

\begin{theorem}\label{app:theY}
$\mathcal D(-1,1)$ is dense in $Y$.
\end{theorem}

\begin{proof}
Note that $\omega^2 \mathbb P \subset H^1_0(-1,1)\subset Y$ by
Proposition \ref{app:propY2}. By Proposition \ref{app:dens3} the
first subset is dense in $Y$. Moreover, $\mathcal D(-1,1)$ is dense
in $H^1_0(-1,1)$ with the norm of this last space.
\end{proof}

Take now $m\ge 1$. As a consequence of Theorem \ref{app:theY}, the
problem of finding $\lambda \in \mathbb R$ and non--trivial $u \in
Y$ such that
\begin{equation}\label{app:Legendreeq}
\mathcal L_m u=\lambda u \quad
\mbox{in $(-1,1)$}
\end{equation}
is easily seen to be equivalent to finding $\lambda \in \mathbb R$
and non--trivial $u\in Y$ satisfying
\[
(u,v)_m= \lambda (u,v)_{L^2(-1,1)} \quad \forall v \in Y,
\]
where
\[
(u,v)_m:=\int_{-1}^1 \Big(\omega^2 u'v'+ \frac{m^2}{\omega^2}  u
v\Big).
\]
Consider now for $m\ge 1$ the set of functions $\{Q_n^m\,:\, n\ge m\}$. They form
an orthonormal basis of $L^2(-1,1)$. On the other hand, because
$Q_n^m$ solves Legendre's equation \eqref{app:Legendreeq} with
$\lambda=n(n+1)$, then
\[
(Q_n^m, Q_\ell^m)_m =n(n+1) \delta_{n,\ell}.
\]
This means that $\{Q_n^m\,:\, n\ge m\}$ is orthogonal in $Y$. By
Proposition \ref{app:dens3}, $\mathrm{span}\,\{ Q_n^m\,:\, n \ge
m\}=\omega^m \mathbb P$ is dense in $Y$, so $\{ Q_n^m\,:\, n\ge m\}$
is a complete orthogonal set in $Y$ and we can characterize
\[
Y=\{ u \in L^2(-1,1)\,:\, \sum_{n=m}^\infty
 n(n+1)\, |(u,Q_n^m)_{L^2(-1,1)}|^2 < \infty\}.
\]
The injection of $Y$ into $L^2(-1,1)$ is therefore compact. Let
$f\in L^2(-1,1)$ and let $G_mf:=u$ be the solution of
\begin{equation}\label{defGm}
u \in Y, \qquad \mbox{s.t.}\qquad(u,v)_m=(f,v)_{L^2(-1,1)} \quad
\forall v \in Y.
\end{equation}
Standard arguments show that $G_m$ is self--adjoint, compact and
injective. Its Hilbert-Schmidt decomposition is given by the
orthonormal system $\{ Q_n^m\,:\, n \ge m\}$. Actually
\begin{equation}\label{defGm2}
G_m f=\sum_{n=m}^\infty \frac1{n(n+1)} (f,Q_n^m)_{L^2(-1,1)} Q_n^m.
\end{equation}
The set $Y$ can be characterized as the range of $G_m^{1/2}$. We
thus have proved the following theorem:

\begin{theorem}\label{theoremY} For all $m\ge 1$, $H_m^1=Y$. However, $H_0^1 \neq
Y$.
\end{theorem}

\begin{proof} The first assertion has already been proved.
Since $\mathbb P_0 \subset \mathbb P \subset H_0^1$, but $\mathbb
P_0 \cap Y=\{0\}$, it is clear that $H_0^1$ is different from the
rest of the spaces.
\end{proof}

The positive and negative powers of $G_m$ define a new sequence of
Hilbert scales, one for each value of $m\ge 1$. Since
$(n(n+1))^{1/2} \approx 2n+1$, as explained at the beginning of
Section \ref{app:sec1}, these Hilbert scales are just $H_m^s$
with equivalent but different norms: note that for positive values
of $s$, we have obtained the spaces
\[
\mathcal R(G_m^{s/2})= \{ u\in L^2(-1,1)\,:\, \sum_{n=m}^\infty
\big( n(n+1)\big)^{s} |(u,Q_n^m)_{L^2(-1,1)}|^2 < \infty\}
\]
which are obviously the same as $H_m^s$.

\begin{corollary}
For all $m, m'\ge 1$ and $-1\le s\le 1$, $H_m^s= H_{m'}^s$.
\end{corollary}

\begin{proof} Since $H_m^0=L^2(-1,1)$ and $H_m^1=Y$ for all $m\ge
1$, by interpolation we can prove that the scales $H_m^s$ are
independent of $m$ for $0\le s\le 1$ (they have different but
equivalent norms). For $-1\le s<0$ the result follows by duality.
\end{proof}

\section{The second set of identifications}\label{section:Z}

\begin{proposition}\label{Hm2RGm} For all $m\ge 1$
\begin{equation}\label{app:Hm2}
H_m^2 = \{ u\in Y\,:\, \mathcal L_m u \in L^2(-1,1)\}
\end{equation}
and there exist $c_m, C_m >0$ such that
\[
c_m \| u\|_{m,2}\le \| \mathcal L_m u\|_{L^2(-1,1)}\le C_m \|
u\|_{m,2}, \qquad \forall u\in H_m^2.
\]
Moreover, $H^2_1 \neq H_m^2$ for all $m\ge 2$.
\end{proposition}

\begin{proof} Using the fact that $H_m^2=\mathcal R(G_m)$ and the
definition of $G_m$ given by solving problems \eqref{defGm},
\eqref{app:Hm2} follows readily. Since we can use the image norm of
$\mathcal R(G_m)$ as an equivalent norm, we only need to prove that
\[
\| \mathcal L_m u\|_{L^2(-1,1)}^2 = \sum_{n=m}^\infty \big(
n(n+1)\big)^2 |(u,Q_n^m)_{L^2(-1,1)}|^2, \qquad \forall u\in
\mathcal R(G_m).
\]
This can be easily done in $\omega^m \mathbb P=\mathrm{span}\,\{
Q_n^m\,:\, n \ge m\}$ using the eigenvalue property for the Legendre
functions \eqref{Qnm}, their $L^2(-1,1)$ orthonormality and the
density of $\omega^m \mathbb P$ in $\mathcal R(G_m)$.

From \eqref{app:Hm2} it follows that $\omega\in H_1^2$ ($\omega
\mathbb P$ is dense in $H_1^2$) but $\omega \not\in H_m^2$ for $m\ge
2$. To see that notice that if $u \in H_m^2 \cap H_1^2$, we would
need that $\omega^{-2}u\in L^2(-1,1)$ (compare the differential
characterizations \eqref{app:Hm2} of these spaces), which is not the
case for $u=\omega$. \end{proof}

Proposition \ref{Hm2RGm} allows us to recommence the construction of
the Hilbert scales for $m\ge 1$ in a different way, using the
unbounded operator $\mathcal L_m: D(\mathcal L_m) \subset L^2(-1,1)
\to L^2(-1,1)$ with domain
\[
D(\mathcal L_m):= \{ u\in Y\,:\, \mathcal L_m u\in L^2(-1,1)\},
\]
with the operator $\mathcal L_m$ applied in the sense of
distributions. Once $\mathcal L_m$ is shown to be selfadjoint, the
Hilbert scale (or Sobolev tower as explained in \cite{EngelNagel},
or left--definite spectral sets) can be constructed again. Note that
is relevant that we demand $u\in Y$ as part of the conditions for
a function to be in the domain of the operator but once that is
done, we only require the image of the differential operator to be
in $L^2(-1,1)$. In this sense, the cases $m\ge 1$ differ in an
essential way from the case $m=0$, where some additional boundary
conditions appear in the domain of the differential operator that
are not in the `energy space' (see Appendix B and
\cite{BrLiTuWe:2009} and references therein).

Consider the space
\[
Z:= \{ u :(-1,1) \to \mathbb R\,:\, \omega^{-2} u, u',
\omega^2u''\in L^2(-1,1)\}
\]
endowed with its natural norm
\[
\| u\|_Z:=\left( \int_{-1}^1 \omega^{-4}|u|^2+|u'|^2+\omega^4
|u''|^2\right)^{1/2}.
\]
Note that the elements of $Z$ are in
$H^2_{\mathrm{loc}}(-1,1)\subset \mathcal C^1(-1,1)$ (the argument
is identical to that of Proposition \ref{prop:YH1loc}). The
remainder of this section is going to be devoted to proving the
following result:

\begin{theorem}\label{theoremZ} For all $m\ge 2$, $H_m^2=Z$ with equivalent norms.
Moreover $Z \subset H^2_0\cap H^2_1$ (it is a strict subset) and
both $H^2_1\setminus H^2_0$ and $H^2_0\setminus H_1^2$ are
non--empty.
\end{theorem}

We start the process of proving Theorem \ref{theoremZ} by showing
some key properties.

\begin{proposition}\label{propZ1} For $m\ge 2$, $\omega^m \mathbb P \subset Z \subset H^1_0(-1,1) \subset
Y$. The injection of $Z$ in $H^1_0(-1,1)$ is continuous.
\end{proposition}

\begin{proof}
The assertion $\omega^m \mathbb P\subset Z$ for $m \ge 2$ follows
directly from the definition of $Z$. Note also that $Z \subset
H^1(-1,1)\subset \mathcal C[-1,1]$. From the definitions of $Z$ and
$Y$ it is clear that $Z\subset Y$, but Proposition \ref{app:propY1}
shows that elements of $Y$ vanish in $\pm1$, which proves that
$Z\subset H^1_0(-1,1)$.
\end{proof}

\begin{lemma}\label{lemmaZ} Let $u \in Z$. Then
$\omega^{-1}u,\,  \omega\,u' $ and $ u\,u'$ are in $\mathcal C_0$.
\end{lemma}

\begin{proof} By Proposition \ref{propZ1}, $Z \subset H^1_0(-1,1)$
and then by Lemma \ref{lemma:H10}(a), $\omega^{-1}u\in \mathcal
C_0$.

Consider now the function $v:= \omega^2 u'$. It is simple to check
that $v\in H^1(-1,1)\subset \mathcal C[-1,1]$ by the conditions that
define $Z$. Also $\omega^{-2} v=u'\in L^2(-1,1)$. These two
properties imply that $v(\pm 1)=0$ and therefore $v \in
H^1_0(-1,1)$. We now apply Lemma \ref{lemma:H10}(a) to prove that
$\omega\, u'=\omega^{-1} v \in \mathcal C_0$.

Finally $u\,u'= (\omega^{-1}u)(\omega\,u')$ and $\mathcal C_0$ is an
algebra and therefore the last assertion is a simple consequence of
the first two.
\end{proof}

\begin{proposition}\label{LmZ} For all $m\ge 2$, there exist constants $C_m,
c_m>0$ such that
\[
c_m\| u\|_Z\le \| \mathcal L_m u\|_{L^2(-1,1)}\le C_m \|u\|_Z,
\qquad \forall u \in Z.
\]
For $m=1$ there holds an upper bound $\| \mathcal
L_1u\|_{L^2(-1,1)}\le C_1 \| u\|_Z$ for all $u \in Z$.
\end{proposition}

\begin{proof} It is simple to check that if $u\in Z$, then
$\mathcal L_m u \in L^2(-1,1)$. Elementary computations show that
\begin{eqnarray}\nonumber
\| \mathcal L_m u\|_{L^2(-1,1)}^2 &=& m^4 \int_{-1}^1 \omega^{-4}
|u|^2 + \int_{-1}^1 | (\omega^2 u')'|^2 -
2m^2\int_{-1}^1 \omega^{-2} u (\omega^2 u')'\\
&=& m^4 \int_{-1}^1 \omega^{-4} |u|^2 +\int_{-1}^1 4 t^2 |u'|^2
+\int_{-1}^1 \omega^4 |u''|^2- 4\int_{-1}^1 t
\omega^2 u'u''\nonumber\\
& & - 2m^2 \int_{-1}^1 u\,u'' + 4m^2 \int_{-1}^1 t
\omega^{-2}u\,u'.\label{Lm2}
\end{eqnarray}
We are now going to apply Lemma \ref{intbyparts} three times: (a)
$2t \omega^2 (u')^2 \in L^1(-1,1) \cap \mathcal C_0$ by Lemma
\ref{lemmaZ} and in the decomposition
\[
(2t \omega^2 (u')^2)'= 4t \omega^2 u'u''+ (2-6t^2) (u')^2,
\]
both terms are in $L^1(-1,1)$; (b) we have a similar behavior of $
u\,u'\in L^1(-1,1)\cap \mathcal C_0$ (again by Lemma \ref{lemmaZ})
with
\[
(u\,u')' = |u'|^2+u\,u''\,;
\]
(c) finally $t\omega^{-2}u^2\in L^1(-1,1) \cap \mathcal C_0$ (Lemma
\ref{lemmaZ}) and
\[
(t\omega^{-2}u^2)'=2 t \omega^{-2}u\,u' +(1+t^2) \omega^{-4} u^2.
\]
The above justifies using integration by parts in the last three
terms of \eqref{Lm2} to obtain
\begin{eqnarray}\nonumber
\| \mathcal L_m u\|^2_{L^2(-1,1)} &=& m^4 \int_{-1}^1 \omega^{-4}
|u|^2 +\int_{-1}^1 4 t^2 |u'|^2 +\int_{-1}^1 \omega^4
|u''|^2+\int_{-1}^1 (2-6t^2) |u'|^2\\
& &  +2m^2 \int_{-1}^1 |u'|^2 - 2m^2 \int_{-1}^1 (1+t^2)
\omega^{-4}|u|^2\nonumber \\
&=& \int_{-1}^1 q_m\, \omega^{-4}|u|^2+\int_{-1}^1 p_m |u'|^2
+\int_{-1}^1 \omega^4 |u''|^2\label{Lmvarphi}
\end{eqnarray}
with
\[
p_m(t):= 2(1+m^2-t^2), \qquad q_m(t):= m^2(m^2-2-2t^2).
\]
Note that
\[
2m^2 \le p_m(t) \le 2+2m^2, \qquad m^2 (m^2-4) \le q_m(t)\le m^2
(m^2-2),\qquad -1\le t \le 1.
\]
This proves the result for any $m \ge 3$ and only the upper bound
for $m=2$ (the constant for the lower bound of $q_m$ cancels). In
this last case we apply integration by parts only to two of the
three last terms of \eqref{Lm2}. Our starting point for the lower
bound is then
\[
\| \mathcal L_2 u\|^2_{L^2(-1,1)} = 16 \int_{-1}^1 \omega^{-4}|u|^2
+\int_{-1}^1 p_2 |u'|^2 +\int_{-1}^1 \omega^4 |u''|^2+16\int_{-1}^1
t \omega^{-2}u\,u'
\]
Applying the inequality
\[
ab \le \frac13\, a^2+\frac34 b^2
\]
we can bound
\[
16 \int_{-1}^1 (t\, u)(\omega^{-2} u') \ge -\frac{16}3 \int_{-1}^1
t^2 |u'|^2-12 \int_{-1}^1 \omega^{-4}|u|^2
\]
and thus
\begin{eqnarray*}
\| \mathcal L_2 u\|^2 &\ge & 4\int_{-1}^1 \omega^{-4}|u|^2
+\int_{-1}^1 \Big( 10 - 2t^2-\frac{16}3 t^2\Big) |u'|^2 +\int_{-1}^1
\omega^4 |u''|^2
\\
& \ge & 4\int_{-1}^1 \omega^{-4}|u|^2 +\frac83\int_{-1}^1 |u'|^2
+\int_{-1}^1 \omega^4 |u''|^2.
\end{eqnarray*}
This completes the proof for $m=2$. The upper bound for $m=1$ is a
direct consequence of \eqref{Lmvarphi}. \end{proof}

\paragraph{Proof of Theorem \ref{theoremZ}.}
Using Propositions \ref{propZ1} and \ref{LmZ} and the
characterization of $H^2_m$ of Proposition \ref{Hm2RGm} it follows
that $Z \subset H^2_m$ with continuous injection. By Propositions
\ref{Hm2RGm} and \ref{LmZ}, $Z$ is a closed subspace of $H^2_m$ for
every $m\ge 2$. However, by Proposition \ref{propZ1}, we know that
$\omega^m \mathbb P\subset Z$ and at the same time $\omega^m \mathbb
P$ is dense in $H^2_m$ by definition. This proves that $Z=H_m^2$ for
$m\ge 2$ with equivalent norms.

Note that Propositions \ref{Hm2RGm} and \ref{LmZ} (the last
assertion of this one) prove that $Z\subset H_1^2$. To see that $Z
\subset H^2_0$, we need to characterize $H^2_0$ as in
\eqref{Littlejohn}. This will be done in the next section (it is a
particular case of Theorem \ref{app:the2}), although the result
follows from results in \cite{BrLiTuWe:2009} and related references.

It is simple to see that $p(t)\equiv 1$ belongs to $H^2_0$ but not
to $Y$ (and therefore not to $Z\subset Y$). Also $\omega \in H^2_1$,
but $\omega \not\in Z$ and $\omega \not\in H^2_0$ by
\eqref{Littlejohn}.  \hfill$\Box$

\section{Weighted Sobolev space characterization}\label{oldsection4}

The following section is devoted to giving a characterization of the
spaces $H_m^k$ for positive integer values of $k$ (and all $m$) as
weighted Sobolev spaces. This section is independent of Sections
\ref{section:Y} and \ref{section:Z} and does not use any result that
appears therein.

Consider the spaces
\[
X_m^k:= \{ u\in L^2(-1,1)\,:\, \omega^{m+k}(\omega^{-m}u)^{(k)} \in
L^2(-1,1)\},
\]
endowed with their natural norms:
\[
\| u\|_{X_m^k}:= \left( \int_{-1}^1 \omega^{2m+2k}
|(\omega^{-m}u)^{(k)}|^2 +\int_{-1}^1 |u|^2\right)^{1/2}.
\]
It is  simple to observe that due to the fact that $\omega$ is
bounded in $[-1,1]$, multiplication by $\omega$ defines a bounded
linear operator from $X_m^k$ into $X_{m+1}^k$ for all $k$ and $m$.
The aim of this section is the proof of the following theorem:

\begin{theorem}\label{app:the2}
For all $m\ge 0$ and $k\ge 1$,
\[
H_m^k = X_m^k
\]
and there exist $C_{m,k}>0$ such that for all $u \in H^k_m$
\[
C_{m,k} \| u\|_{m,k}^2\le (m+{\textstyle\frac12})^{2k}\int_{-1}^1
|u|^2 +\int_{-1}^1 \omega^{2(m+k)}| (\omega^{-m}u)^{(k)}|^2 \le
2^{1-2k} \| u\|_{m,k}^2.
\]
\end{theorem}

As in Section \ref{section:Z}, we first prove some key results that
will allow us to show the identifications of spaces of Theorem
\ref{app:the2} at the end of this section.

\begin{proposition}\label{app:newX0} For all $k \ge 1$ and $m \ge 0$
\[
X_m^k \subset H^k_{\mathrm{loc}}(-1,1) \subset \mathcal C^{k-1}
(-1,1).
\]
Moreover, if $v\in H^k(-1,1)$, then $\omega^m v\in X_m^k$.
\end{proposition}

\begin{proof} The first part uses the same arguments as the ones in
Proposition \ref{prop:YH1loc}. The second part is straightforward.
\end{proof}

\begin{proposition}\label{prop:charactXMk} For all $m\ge 0$ and $k\ge 1$, $X_m^k\subset
X_m^{k-1}$ with continuous injection. Therefore
\[
X_m^k= \{ u: (-1,1) \to \mathbb R\,:\,
\omega^{m+\ell}(\omega^{-m}u)^{(\ell)}\in L^2(-1,1), \quad 0 \le
\ell \le k\}.
\]
\end{proposition}

\begin{proof}
Given $u \in X_m^k$, we consider the function
\[
v:= \omega^{m+k} (\omega^{-m} u)^{(k)}\in L^2(-1,1).
\]
There holds
$(\omega^{-m}u)^{(k-1)}\in H^1_{\mathrm{loc}}(-1,1) \subset \mathcal
C(-1,1)$ and we can bound
\[
\jmath (u):= (\omega^{-m}u)^{(k-1)}(0), \qquad |\jmath(u)| \le
C_{m,k} \|u\|_{X_m^k}.
\]
Therefore
\begin{equation}\label{app:partsint}
(\omega^{-m}u)^{(k-1)}(t)=\int_0^t (\omega^{-m}u)^{(k)}(s) \mathrm d
s+\jmath(u)= \int_0^t \omega^{-m-k}(s) v(s) \mathrm d s+\jmath(u).
\end{equation}
For all $t \in (-1,1)$
\begin{eqnarray*}
\Big| \omega^{m+k-1}(t)\int_0^t \omega^{-m-k}(s)v(s) \mathrm d
s\Big|^2 & \le &  \Big| \omega^{2m+2k-2}(t) \int_0^t
\omega^{-2m-2k}(s) \mathrm d s\Big|\, \Big|\int_0^t |v(s)|^2
\mathrm d s\Big|\\
& = & | g_{m+k-1}(t)| \Big|\int_0^t |v(s)|^2\mathrm d s\Big| \le |
g_{m+k-1}(t)| \| v\|_{L^2(-1,1)}^2,
\end{eqnarray*}
$g_{m+k-1}\in L^1(-1,1)$ being one of the functions of Lemma
\ref{app:lemmag}. Therefore, using \eqref{app:partsint}, it follows
that
\[
\omega^{m+k-1}(\omega^{-m}u)^{(k-1)}\in L^2(-1,1)
\]
and its norm is controlled by the one of $u$ in $X_m^k$.
\end{proof}

\begin{corollary} Let $m,k \ge 0$. Multiplication by a fixed $v \in
\mathcal C^\infty[-1,1]$ is a linear bounded operator from $X_m^k$
to itself.
\end{corollary}

\begin{proof} Let $v \in \mathcal C^\infty[-1,1]$ and $u \in
L^2(-1,1)$. Then
\[
\omega^{m+k} (\omega^{-m} u\, v)^{(k)}= \sum_{j=0}^k \binom{k}{j}
\omega^{k-j} v^{(k-j)}\, \omega^{m+j} (\omega^{-j} u)^{(j)}
\]
and the result is a direct consequence of Proposition
\ref{prop:charactXMk}. Note that the result also holds if $v\in
W^{\infty,k}(-1,1):= \{ v:(-1,1) \to \mathbb R\,:\, v^{(j)}\in
L^\infty(-1,1), \quad 0\le j\le k\}$.
\end{proof}

\begin{lemma}\label{app:regul1}
If $u \in X_m^k$ with $k\ge 1$ and $m\ge 0$, then $\omega^{2m+2k}
(\omega^{-m}u)^{(k-1)}\in H^1_0(-1,1)$.
\end{lemma}

\begin{proof} Note that if we prove that
$v:=\omega^{m+k+1}(\omega^{-m} u)^{(k-1)}\in H^1_0(-1,1)$, then by
Lemma \ref{lemma:H10}(b) the result follows readily.

First of all, it is clear that $v \in L^2(-1,1)$ and that
\[
v'=-(m+k+1)\,t\, \omega^{m+k-1}(\omega^{-m}
u)^{(k-1)}+\omega^{m+k+1} (\omega^{-m} u)^{(k)}\in L^2(-1,1),
\]
that is, $v \in H^1(-1,1)\subset \mathcal C[-1,1]$. Besides,
$\omega^{-2} v \in L^2(-1,1)$ and therefore $v(\pm 1)=0$, which
proves the result.
\end{proof}

\begin{lemma}\label{app:regul3} If $u\in X_m^k$ with $m\ge 0$ and $k\ge 1$, then
\[
\int_{-1}^1 (\omega^{-m}u)^{(k)} \,\omega^{2m+2k}\, p =
(-1)^k\int_{-1}^1 \omega^{-m}u \,(\omega^{2m+2k} p)^{(k)}, \qquad
\forall p \in \mathbb P.
\]
\end{lemma}

\begin{proof} The result is true for $p\in \mathcal C^k[-1,1]$ but
the argument is simpler with polynomials. Note first that by Lemma
\ref{app:regul1},
\[
\omega^{2m+2k}(\omega^{-m}u)^{(k-1)} p \in H^1_0(-1,1)
\]
and that in the expression
\[
\Big( (\omega^{-m}u)^{(k-1)} \omega^{2m+2k} p\Big)' =
(\omega^{-m}u)^{(k)} \omega^{2m+2k}p +(\omega^{-m}u)^{(k-1)}
(\omega^{2m+2k} p)'
\]
both terms on the right hand side are integrable, so can apply
integration by parts to prove that
\[
\int_{-1}^1 (\omega^{-m}u)^{(k)} \omega^{2m+2k}p = - \int_{-1}^1
(\omega^{-m} u)^{(k-1)} (\omega^{2m+2k} p)'.
\]
However, $(\omega^{2m+2k} p )'=\omega^{2m+2k-2} q$ with $q\in
\mathbb P$ and $u \in X_m^{k-1}$. These facts allow us to apply the
same argument again to obtain, when $k\ge 2$
\begin{eqnarray*}
\int_{-1}^1 (\omega^{-m}u)^{(k)}\omega^{2m+2k} p &=& -\int_{-1}^1
(\omega^{-m} u)^{(k-1)} \omega^{2m+2k-2} q \\
&  = & \int_{-1}^1 (\omega^{-m}u)^{(k-2)} (\omega^{2m+2k-2} q)' =
\int_{-1}^1 (\omega^{-m}u)^{(k-2)} (\omega^{2m+2k} p)''.
\end{eqnarray*}
The statement follows by induction.
\end{proof}

\begin{theorem}\label{wmPdense}
For all $m\ge 0$ and $k\ge 1$, $\omega^m \mathbb P$ is dense in
$X_m^k$.
\end{theorem}

\begin{proof}
Assume that $u\in X_m^k$ is orthogonal to all elements of the form
$\omega^m p$ with $p$ a polynomial. This means that
\[
\int_{-1}^1 \omega^{2m+2k}(\omega^{-m}u)^{(k)} p^{(k)}+\int_{-1}^1
u\, \omega^m p =0, \qquad \forall p \in \mathbb P.
\]
We can apply Lemma \ref{app:regul3} to prove that
\begin{equation}\label{app:regul4a}
(-1)^k \int_{-1}^1 \omega^{-m}u \,(\omega^{2m+2k}
p^{(k)})^{(k)}+\int_{-1}^1 u\, \omega^m p=0, \qquad \forall p \in
\mathbb P.
\end{equation}
Consider now the operator
\[
\mathcal L_{m,k}p:= (-1)^k \omega^{-2m} (\omega^{2m+2k}
p^{(k)})^{(k)} + p
\]
This is a linear differential operator with polynomial coefficients
and maps $\mathbb P_n$ to itself. Furthermore, if $p\in \mathbb P$
and $\mathcal L_{m,k}p=0$, then taking $u=\omega^m p$ in \eqref{app:regul4a}
we derive
\[
0 = \int_{-1}^1 (\omega^{2m} p) \, (\mathcal L_{m,k}p)= \int_{-1}^1
\omega^{2m+2k} |p^{(k)}|^2+\int_{-1}^1 \omega^{2m}|p|^2.
\]
This means that $\mathcal L_{m,k}: \mathbb P \to \mathbb P$ is
injective and does not increase the degree. Therefore $\mathcal
L_{m,k}:\mathbb P \to \mathbb P$ is a bijection. Condition
\eqref{app:regul4a} is thus equivalent to
\[
\int_{-1}^1 \omega^m u \,  q=0, \qquad \forall q \in \mathbb P,
\]
which implies that $u=0$. Therefore $\omega^m \mathbb P$ is dense in
$X_m^k$.
\end{proof}

\paragraph{Proof of Theorem \ref{app:the2}.}
Let $\|\hspace{2pt}\cdot\hspace{2pt}\|_{m,k,\ast}$ be the norm that
appears in the inequality of the statement and
$(\hspace{2pt}\cdot\hspace{2pt},\hspace{2pt}\cdot\hspace{2pt})_{m,k,\ast}$
the associated inner product. Note that $\{P_n^m\,:\, n\ge m\}$ is
orthogonal in $H_m^k$ and that since
\[
\omega^{m+k} (\omega^{-m} P_n^m)^{(k)}=\omega^{m+k}
P_n^{(m+k)}=\underline P_n^{m+k}:=\left\{ \begin{array}{ll}
P_n^{m+k}, & \mbox{if $n \ge
m+k$},\\[1.5ex] 0 , & \mbox{if $m+k > n\ge m$},\end{array}\right.
\]
then
\[
(P_n^m,P_\ell^m)_{m,k,\ast}=(m+{\textstyle\frac12})^{2k} \int_{-1}^1
P_n^m\,P_\ell^m + \int_{-1}^1\underline P_n^{m+k}\,\underline
P_\ell^{m+k} ,
\]
so these functions are also orthogonal in $X_m^k$.

For all $n\ge m$
\[
\| P_n^m\|_{m,k}^2 = c_{n,m}^{-2} \| Q_n^m\|_{m,k}^2 = c_{n,m}^{-2}
(2n+1)^{2k}.
\]
If $m+k > n\ge m$, then
\[
\| P_n^m\|_{m,k,\ast}^2= (m+{\textstyle\frac12})^{2k} c_{n,m}^{-2},
\]
whereas for $n\ge m+k$
\[
\| P_n^m\|_{m,k,\ast}^2= (m+{\textstyle\frac12})^{2k}c_{n,m}^{-2} +
c_{n,m+k}^{-2}=c_{n,m}^{-2} \Big( (m+{\textstyle\frac12})^{2k}+
\left( \frac{c_{n,m}}{c_{n,m+k}}\right)^2\Big).
\]
Using Lemma \ref{app:lemma2} we can bound
\begin{eqnarray*}
\frac{(m+1)!}{(m+k+1)!} \frac{(2n+1)^{2k}}{2^{2k}}\,  \le
\left(\frac{c_{n,m}}{c_{n,m+k}}\right)^2 & \le &
(m+{\textstyle\frac12})^{2k} +
\left(\frac{c_{n,m}}{c_{n,m+k}}\right)^2 \\
& \le & 2\,\frac{(2n+1)^{2k}}{2^{2k}} ,
\end{eqnarray*}
which translates into
\[
C_{m,k} \| P_n^m\|_{m,k}^2 \le \|
P_n^m\|_{m,k,\ast}^2 \le \frac{2}{2^{2k}} \| P_n^m\|_{m,k}^2, \qquad
n\ge m.
\]
where
\[
 C_{m,k}:=\min\Big\{\frac{(m+1)!}{(m+k+1)!}
\frac1{2^{2k}}, \frac{(m+1/2)^{2k}}{(2m+2k-1)^{2k}}\Big\}.
\]

Note that $\omega^m \mathbb P=\mathrm{span}\, \{ P_n^m\,:\, n \ge
m\}$ is dense in both $X_m^k$ and $H_m^k$. Using Lemma
\ref{app:lemma4} we prove that $X_m^k$ can be identified with
$H_m^k$ and that the inequalities for the norms of $P_n^m$ can be
extended to any element of the spaces. \hfill$\Box$

Note that as a consequence of the fact that $H_m^1=Y$ for all $m\ge
1$, it follows that
\[
\{ u\in L^2(-1,1)\,:\, \omega^{m+1}(\omega^{-m}u)'\in L^2(-1,1)\} =
\{ u\,:\, \omega^{-1}u, \omega\,u' \in L^2(-1,1)\},
\]
as long as $m\neq 0$.

\section{Further properties}\label{sec:further}

\begin{theorem}[Embedding theorem]\label{app:the1}
Let $u \in H^s_m$ with $s > m+2k+1$. Then $\omega^{-m}u \in \mathcal
C^k[-1,1]$. Moreover, there exist $C_{m,s}>0$ such that
\[
\max_{0\le \ell\le k}  \|
(\omega^{-m}u)^{(\ell)}\|_{L^\infty(-1,1)}\le C_{m,s}\| u\|_{m,s},
\qquad \forall u\in H^s_m.
\]
\end{theorem}

\begin{proof} Decomposing $u$ in $\{ Q_n^m \,:\, n\ge m\}$, which is an
orthonormal basis of $L^2(-1,1)$, we observe that
\[
u= \sum_{n=m}^\infty u_n\, Q_n^m = \omega^m \sum_{n=m}^\infty u_n\,
c_{n,m}\, P_n^{(m)}, \qquad u_n:= (u,Q_n^m)_{L^2(-1,1)}.
\]
Using Lemmas \ref{app:lemma1} and \ref{app:lemma3}, we can bound
term by term
\begin{eqnarray*}
| c_{n,m} P_n^{(m+k)} (t)| & \le & C_m \, (2n+1)^{1/2-m}D_{m+k}
\, (2n+1)^{2m+2k}\\
& =& C_m\, D_{m+k}\, (2n+1)^{1/2+m+2k}, \qquad \forall t \in [-1,1].
\end{eqnarray*}
Note that if $s > r+1$, then $2(r-s)+1 < -1$ and
\[
\sum_{n=1}^\infty | f_n| n^{r+1/2}\le \left(\sum_{n=1}^\infty
|f_n|^2 n^{2s}\right)^{1/2} \left( \sum_{n=1}^\infty
n^{2(r-s)+1}\right)^{1/2} =: E_{s-r} \left(\sum_{n=1}^\infty |f_n|^2
n^{2s}\right)^{1/2}.
\]
Therefore
\begin{eqnarray*}
\sum_{n=1}^\infty| u_n c_{n,m} P_n^{(m+k)} (t)|& \le & C_m\,
D_{m+k} \,\sum_{n=1}^\infty | u_n| \, (2n+1)^{m+2k+1/2} \\
& \le & C_m\, D_{m+k}\, E_{s-m-2k} \left(\sum_{n=1}^\infty
|u_n|^2 (2n+1)^{2s}\right)^{1/2}\\
&=&  C_m\, D_{m+k}\, E_{s-m-2k} \|u\|_{m,s},
\end{eqnarray*}
for all $t \in [-1,1]$ and $s > 2k+m+1$. Therefore the series
\[
\sum_{n=1}^\infty u_n c_{n,m} P_n^{(m+k)}
\]
converges uniformly to continuous functions, which proves that
$\omega^{-m}u$ has $k$ continuous derivatives in $[-1,1]$.
\end{proof}

\begin{proposition}\label{prop:intersection}
\[
\bigcap_{k=1}^\infty H_m^k = \{ \omega^m\,g \,:\, g \in \mathcal
C^\infty[-1,1]\}.
\]
\end{proposition}

\begin{proof} By Theorem \ref{app:the2}, if $g \in \mathcal
C^k[-1,1]$, then $u:=\omega^m g\in H_m^k$. Therefore elements of the
form $\omega^m\, g$ with $g \in \mathcal C^\infty[-1,1]$ belong to
$H_m^k$ for all $k$. The converse statement is a direct consequence
of Theorem \ref{app:the1}.
\end{proof}

The following result gives a simple identity regarding the
$m-$weighted forms of the natural norms of $X_m^1$ and $Y$. Thanks
to it we will be able to prove a first set of inclusions of the
spaces $X_m^k$ for different values of $m$.

\begin{proposition}\label{prop:identY} For all $m \ge 1$ and $u\in H_m^1=Y$
\[
m(m+1) \int_{-1}^1 |u|^2 +\int_{-1}^1
\omega^{2m+2}|(\omega^{-m}u)'|^2 = m^2 \int_{-1}^1\omega^{-2}|u|^2
+\int_{-1}^1 \omega^2 |u'|^2.
\]
\end{proposition}

\begin{proof} Note that $H_m^1=X_m^1=Y$ for all $m\ge 1$ and we can
use many integral properties of $u\in H_m^1$. In particular $u\in
H^1_{\mathrm{loc}}(-1,1)$ and we can apply the product rule
$(u^2)'=2u\,u'$. Therefore
\begin{eqnarray*}
m(m+1) u^2 &+& \omega^{2m+2} ((\omega^{-m}u)')^2 \\
&=& m(m+1) u^2 +m^2 t^2 \omega^{-2}u^2 +\omega^2 (u')^2 + 2
m t u\,u'\\
&=& u^2 m^2(1+\omega^{-2} t^2) + \omega^2 (u')^2 +
m(u^2+2t\,u\,u')\\&=& m^2 \omega^{-2}u^2 +\omega^2 (u')^2 + m
(t\,u^2)'.
\end{eqnarray*}
Using the arguments of the proof of Proposition \ref{app:propY1}, we
can show that
\[
\int_{-1}^1 (t\,u^2)' =0, \qquad \forall u \in Y,
\]
from where we obtain the equality of norms. Note that the result can
also be proved by comparison of the norms of the functions $Q_n^m$
and using Lemma \ref{app:lemma4}.
\end{proof}

\begin{corollary}\label{cor:desHardy}
Let $m\ge 1$. Then
\[
m \|\omega^{m-1} v\|_{L^2(-1,1)}^2\le
(m+1)\|\omega^{m}v\|_{L^2(-1,1)}^2+\frac{1}{m}\|\omega^{m+1}v'\|_{L^{2}(-1,1)}^2,
\qquad \forall v \in H^1(-1,1).
\]
\end{corollary}
\begin{proof}
Take $u :=\omega^{m} v$ and note that $u \in X_{m}^1$ by Proposition
\ref{app:newX0}. Applying Proposition \ref{prop:identY} to $u$, we
obtain
\begin{eqnarray*}
m(m+1)\|\omega^{m}v\|_{L^2(-1,1)}^2+\|\omega^{m+1}v'\|_{L^2(-1,1)}^2
&= & m^2\|\omega^{m-1}v\|^2_{L^2(-1,1)}+\|\omega(\omega^m v)'\|_{L^2(-1,1)}^2\\
& \ge&  m^2\|\omega^{m-1}v\|^2_{L^2(-1,1)},
\end{eqnarray*}
which proves the result.
\end{proof}

\begin{proposition}\label{propm+2} Let $m,k \ge0$. Then $X_{m+2}^k \subset X_m^k$
with continuous injection. The inclusion is strict if $m < k$.
\end{proposition}

\begin{proof} The result is trivial for $k=0$ and has already been
proved for $k=1$, when  $X_m^1=Y$ for $m>0$ and $Y$ is a strict
subspace of $X_0^1$.

The density of $\omega^{m+2}\mathbb P$ in $X_{m+2}^k$ (Theorem
\ref{wmPdense}) reduces the proof to showing that
\begin{equation}\label{mm+2:b}
\| \omega^{m+k} (\omega^{-m} \omega^{m+2} p)^{(k)}\|_{L^2(-1,1)} =
\| \omega^{m+k} (\omega^2 p)^{(k)}\|_{L^2(-1,1)}\le C_{m,k} \|
\omega^{m+2} p\|_{X_{m+2}^k}, \qquad \forall p \in \mathbb P.
\end{equation}
By Proposition \ref{app:newX0} and Theorem \ref{wmPdense} we could
equivalently prove it for $p\in H^k(-1,1)$. Note that
\begin{equation}\label{mm+2:a}
\omega^{m+k} (\omega^2 p)^{(k)} = \omega^{m+k+2} p^{(k)}-2 k
t\omega^{m+k} p^{(k-1)}-k(k-1) \omega^{m+k} p^{(k-2)}.
\end{equation}
The three terms on the right hand side are bounded separately. First
of all
\[
\| \omega^{m+k+2}p^{(k)}\|_{L^2(-1,1)}= \|
\omega^{m+2+k}(\omega^{-m-2} \omega^{m+2} p)^{(k)}\|_{L^2(-1,1)}\le
\| \omega^{m+2} p\|_{X_{m+2}^k}.
\]
The third term of \eqref{mm+2:a} is equally easy to bound
\begin{eqnarray*}
\| \omega^{m+k} p^{(k-2)}\|_{L^2(-1,1)}&=& \| \omega^{m+2+k-2}
(\omega^{-m-2}\omega^{m+2}p)^{(k-2)}\|_{L^2(-1,1)} \\
& \le & \| \omega^{m+2} p\|_{X_{m+2}^{k-2}}\le C_{m,k}^{[1]} \|
\omega^{m+2}p\|_{X_{m+2}^k}
\end{eqnarray*}
by Proposition \ref{prop:charactXMk}. Finally, for the second term
of \eqref{mm+2:a} we use Corollary \ref{cor:desHardy} and
Proposition \ref{prop:charactXMk}:
\begin{eqnarray*}
\| \omega^{m+k} p^{(k-1)}\|_{L^2(-1,1)} & \le & C_{m,k}^{[2]} \Big(
\| \omega^{m+k+1}p^{(k-1)}\|_{L^2(-1,1)}+\|
\omega^{m+k+2}p^{(k)}\|_{L^2(-1,1)}\Big)\\
& \le & C_{m,k}^{[2]} \Big( \| \omega^{m+2} p\|_{X_{m+2}^{k-1}} +\|
\omega^{m+2}p\|_{X_{m+2}^k}\Big)\le C_{m,k}^{[3]} \|
\omega^{m+2}\|_{X_{m+2}^k}.
\end{eqnarray*}
Using these last three bounds and \eqref{mm+2:a} we prove
\eqref{mm+2:b}.

Finally, assume that $m<k$. It is clear that $\omega^m\in X^k_m$
(this is true for all values of $m$ and $k$). We now prove that
$\omega^m \not\in X^k_{m+2}$.  Using Lemma \ref{lemma:derivatives},
it follows that
\[
\omega^{m+2+k}(\omega^{-m-2}\omega^m)^{(k)}= \omega^{m-k} p_{-2,k}
\]
where $p_{-2,k}\in \mathbb P_k$ and $p_{-2,k}(\pm 1)\neq 0$.
However, for $m<k$, the singularities of $\omega^{m-k}$  at $\pm 1$ do not allow
it to be in $L^2(-1,1)$ (see \eqref{app:omega2}) and the result is
proved.
\end{proof}

In the following section we will go further to prove that the
remainder inclusions of the above proposition are just equalities of
sets.

\section{The last set of identifications}\label{sec:lastsection}

For any non-negative integer $k$ we consider the Hilbert space
\[
Z_k:=\Big\{u:(-1,1) \to \mathbb R \, : \,
\omega^{-k+2\ell}u^{(\ell)}\in L^2(-1,1),\quad \ell=0,\ldots,k\Big\}
\]
endowed with its natural norm:
\[
\| u\|_{Z_k}:=\left( \sum_{\ell=0}^k \int_{-1}^1 \omega^{2(2\ell-k)}
| u^{(\ell)}|^2 \right)^{1/2}.
\]
Notice that $Z_1=Y$ and $Z_2=Z$ which have appeared in previous
sections.

\begin{proposition} For all $k\ge 1$, $Z_k \subset Z_{k-1}$ with
continuous injection.
\end{proposition}

\begin{proof} It is a straightforward application of the definition
of the spaces. \end{proof}

The aim of this section is to prove the following result:

\begin{theorem}
\label{theo:lastid} Let $m$ and  $k$ be non negative integers:
\begin{itemize}
\item[{\rm (a)}] If $m\ge k$, then $Z_k=X_m^k=H_m^k$ with equivalent
norms.
\item[{\rm (b)}] $Z_k \subset H_0^k \cap H_1^k \cap \ldots \cap
H_{k-1}^k$ and the inclusion is strict.
\end{itemize}
\end{theorem}

The cases $k=1$ and $k=2$ of Theorem \ref{theo:lastid}(a) were
proved in Theorems \ref{theoremY} and \ref{theoremZ}. The different
assertions of this main theorem will be broken down to a small
collection of properties that we now proceed to state and prove.
Because of Proposition \ref{propm+2} and the equality of the sets
$X_m^k=H_m^k$, when $k\ge 2$, the superset in Theorem
\ref{theo:lastid}(b) is just $H_{k-2}^k \cap H_{k-1}^k$. We will end
the section distinguishing the remaining sets $H_m^k$ with $m<k$.

Note that by interpolation and duality, as a consequence of Theorem
\ref{theo:lastid}(a) we prove that if $m,m'\ge k$, then
$H_m^s=H_{m'}^s$ for all real $s$ such that $|s| \le k$.

\begin{proposition}\label{prop:Zk1} For all $m, k\ge 0$, $Z_k\subset X_m^k=H_m^k$
with continuous injection. Moreover, if $m < k$, the inclusion is
strict.
\end{proposition}

\begin{proof} Note first that
\begin{equation}\label{eq:pfpZk2}
\omega^{m+k} \big( \omega^{-m} u\big)^{(k)} = \sum_{\ell=0}^k
\binom{k}{\ell} \omega^{2k-2\ell+m} (\omega^{-m})^{(k-\ell)}\,
\omega^{2\ell-k} u^{(\ell)}.
\end{equation}
The functions $\omega^{2k-2\ell+m} (\omega^{-m})^{(k-\ell)}$ that
appear in \eqref{eq:pfpZk2} are polynomials and we can bound
\[
\| \omega^{m+k} \big( \omega^{-m} u\big)^{(k)}\|_{L^2(-1,1)} \le
C_{m,k} \sum_{\ell=0}^k \|\omega^{2\ell-k}
u^{(\ell)}\|_{L^2(-1,1)}\le \sqrt{k}\, C_{m,k} \| u\|_{Z_k}.
\]
This inequality proves the first assertion of the result. Note now
that $\omega^m \in X_m^k=H_m^k$ (this follows from the definition
and also from Proposition \ref{prop:intersection}). However, if
$\omega^m \in Z_k$, then $\omega^{m-k}\in L^2(-1,1)$, which requires
(see \eqref{app:omega2}) that $m-k > -1$. Therefore $\omega_m
\not\in Z_k$ if $m \le k-1$ and the inclusion of $Z_k$ in $X_m^k$ is
strict in these cases.
\end{proof}

Note that so far we have proved Theorem \ref{theo:lastid}(b) as well
as one of the inclusions needed to prove Theorem
\ref{theo:lastid}(a).

\begin{lemma}\label{cor:auxZk}
For all $m\ge k$, there exist $C_{m,k}>0$ such that
\[
\|\omega^m u\|_{Z_k}\le C_{m,k}\sum_{\ell =0}^k\|\omega^{m-k+2\ell}
u^{(\ell)}\|_{L^2(-1,1)}, \qquad \forall u \in H^k(-1,1).
\]
Hence, $\omega^m H^k(-1,1)\subset Z_k$ for all $m\ge k$.
\end{lemma}

\begin{proof} The $Z_k-$norm of $\omega^m u$ includes
$L^2(-1,1)-$norms of terms like:
\[
\omega^{-k+2\ell}(\omega^m u)^{(\ell)}=\sum_{j=0}^\ell
\binom{\ell}{j} \omega^{2\ell-2j-m}(\omega^m)^{(\ell-j)}  \omega
^{m-k+2j}u^{(j)}.
\]
However, by \eqref{eq:pfpZk2b}, it follows that
$\omega^{2\ell-2j-m}(\omega^m)^{(\ell-j)} = p_{m,\ell-j} \in \mathbb
P$, which allows us to bound
\[
\|\omega^{-k+2\ell}(\omega^m\,u)^{(\ell)}\|_{L^2(-1,1)}\le
C_{k,\ell,m}\sum_{j=0}^\ell \|\omega^{m-k+2j}u^{(j)}\|_{L^2(-1,1)}.
\]
Summing for $\ell=0,\ldots, k$, the result is proven.
\end{proof}

\begin{proposition}\label{prop:Zk2} For $m\ge k$, $X_m^k \subset Z_k$ with
continuous injection.
\end{proposition}

\begin{proof}
Note that $\omega^m \mathbb P$ is a dense subset of $X_m^k$ (this is
Theorem \ref{wmPdense}) and therefore by Proposition
\ref{app:newX0}, so is $\omega^m H^k(-1,1)$. This density result and
Lemma \ref{cor:auxZk} show that if we are able to prove the
inequalities
\begin{equation}\label{induction}
\sum_{\ell=0}^k \|\omega^{m-k+2\ell} v^{(\ell)}\|_{L^2(-1,1)}\le
C_{m,k}\sum_{\ell=0}^k\|\omega^{m+\ell}v^{(\ell)}\|_{L^2(-1,1)},
\qquad \forall v\in H^k(-1,1)
\end{equation}
for any $m\ge k$ we will have proved the result. We will do this by
induction in the pair $(k,m)$.

The inequality \eqref{induction} is clearly true for $(0,m)$ and any
$m\ge 0$. We now assume that that it holds for the pair $(k,m)$ and
proceed to prove it for $(k+1,m+1)$. However, this is a simple
consequence of Corollary \ref{cor:desHardy} as we now show. Note that
\begin{eqnarray*}
\sum_{\ell=0}^{k+1} \|\omega^{m+1-(k+1)+2\ell}
v^{(\ell)}\|_{L^2(-1,1)} & =& \|
\omega^{m+k+2}v^{(k+1)}\|_{L^2(-1,1)}+\sum_{\ell=0}^k \|
\omega^{m-k+2\ell} v^{(\ell)}\|_{L^2(-1,1)}\\
& \le & \|\omega^{m+k+2}v^{(k+1)}\|_{L^2(-1,1)} +
C_{m,k}\sum_{\ell=0}^k\|\omega^{m+\ell}v^{(\ell)}\|_{L^2(-1,1)}\\
& \le & \|\omega^{m+k+2}v^{(k+1)}\|_{L^2(-1,1)}+C_{m,k}'
\sum_{\ell=0}^{k+1} \| \omega^{m+\ell+1} v^{(\ell)}\|_{L^2(-1,1)}\\
&=  & (1+C_{m,k}')\sum_{\ell=0}^{k+1} \| \omega^{m+\ell+1}
v^{(\ell)}\|_{L^2(-1,1)},
\end{eqnarray*}
where we have applied the induction hypothesis and Corollary
\ref{cor:desHardy} to the functions $v^{(\ell)}\in H^1(-1,1)$
(since $\ell \le k$ and $v\in H^{k+1}(-1,1)$).
\end{proof}

Note that Theorem \ref{theo:lastid}(a) is a direct consequence of
Propositions \ref{prop:Zk1} and \ref{prop:Zk2}.

\begin{proposition}
Let $ 0 \le m,m'\le k$. Then $H_{m'}^k\subset H_{m}^k$ if and only
if $m'-m$ is a non--negative even number. In that case, the
inclusion is strict.
\end{proposition}
\begin{proof}
Note that one of the implications is part of Proposition
\ref{propm+2}. However what we are going to prove is that
\begin{equation}\label{eq:wheresomegam}
\omega^{m'}\in X_m^k \qquad\Longleftrightarrow\qquad\left\{
 \begin{array}{l} m'\ge k,\\
 \mbox{$m'<k$ and $m'-m$ is a non-negative even number}
 \end{array}
 \right.
\end{equation}
By Lemma \ref{lemma:derivatives}
\[
v:=\omega^{m+k}(\omega^{-m}\omega^{m'})^{(k)}= \omega^{m'-k}
p_{m'-m,k}, \qquad p_{m'-m,k}\in \mathbb P_k.
\]
If $m'\ge k$ then  $v\in L^2(-1,1)$.  Consider now that $m'< k$.
Then, if  $m'-m$ is a non-negative even number, $\omega^{-m}\omega^{m'}$ is a polynomial of degree $m'-m<k$ and therefore $v\equiv 0$. Otherwise, $p_{m'-m,k}(\pm 1)\ne 0$ and then the local behavior of $v$ in the vicinity of $\pm 1$ shows that $v\not\in
L^2(-1,1)$. The result is now an easy consequence of
\eqref{eq:wheresomegam}.
\end{proof}

\appendix

\section{Technical lemmas}

\subsection{Functions and bounds}

\begin{lemma}\label{lemma:derivatives} For all integer $k\ge 0$ and $\alpha \in \mathbb R$
\begin{equation}\label{eq:pfpZk2b}
(\omega^\alpha)^{(k)} = p_{\alpha,k}\, \omega^{\alpha-2k}, \qquad
p_{\alpha,k}\in \mathbb P_k.
\end{equation}
Moreover
\begin{equation}\label{palphak}
p_{\alpha,k}(-1)= (-1)^k p_{\alpha,k}(1)= \prod_{j=0}^{k-1}
(\alpha-2j).
\end{equation}
Therefore $p_{\alpha,k}(\pm 1)\neq 0$ unless $\alpha$ is a
non--negative even integer and $k\ge \alpha/2 +1$. In this last case,
$p_{\alpha,k}\equiv 0$.
\end{lemma}

\begin{proof} The result is proved by induction in $k$ (the
assertion being valid for all $\alpha \in \mathbb R$). Note that the
case $k=0$ is trivial (giving $p_{\alpha,0}\equiv 1$) and that the
case $k=1$ is just \eqref{app:omega1}, giving
$p_{\alpha,1}(t)=-\alpha \, t$.

Assume that the result is true for all integer values up to $k$. An
application of the induction hypothesis and \eqref{app:omega1} yield
\begin{eqnarray*}
(\omega^\alpha)^{(k+1)} &=& (-\alpha\, t\, \omega^{\alpha-2})^{(k)}
= (-\alpha \,t) (\omega^{\alpha-2})^{(k)} -\alpha\, k\,
(\omega^{\alpha-2})^{(k-1)} \\
&=& (-\alpha\, t) p_{\alpha-2,k} \omega^{\alpha-2-2k} -\alpha\,k\,
p_{\alpha-2,k-1} \omega^{\alpha-2-2(k-1)}\\
&=& (-\alpha\, t p_{\alpha-2,k}-\alpha \, k\, \omega^2\,
p_{\alpha-2,k-1}) \omega^{\alpha-2(k+1)}
\end{eqnarray*}
which gives the result for $k+1$ as well as the formula
\[
p_{\alpha,k+1}= -\alpha \, t p_{\alpha-2,k}-\alpha\,k\, \omega^2
p_{\alpha-2,k-1}.
\]
In particular $p_{\alpha,k+1}(\pm 1)=(\mp \alpha) p_{\alpha-2,k}(\pm
1)$ and \eqref{palphak} follows by another inductive argument. The
final assertions of the lemma are straightforward.
\end{proof}

\begin{lemma}\label{app:lemmag}
For $n\ge 0$ consider the functions
\[
g_n(t):= \omega^{2n}(t) \int_0^t \omega^{-2n-2}(s) \mathrm d
s=(1-t^2)^n \int_0^t \frac1{(1-s^2)^{n+1}} \mathrm d s.
\]
Then
\[
g_0(t)=\frac12\log\left( \frac{1+t}{1-t}\right)
\]
and $g_n \in \mathcal C^{n-1}[-1,1]$ for $n\ge 1$. In particular
$g_n \in L^1(-1,1)$ for all $n\ge 0$.
\end{lemma}

\begin{proof} Note first the following computation:
\begin{eqnarray*}
\Big( t(1-t^2)^{-n}\Big)' &=& (1-t^2)^{-n} + 2 n t^2
(1-t^2)^{-n-1}\\
&=& (1-t^2)^{-n}+ 2n \Big( (1-t^2)^{-n-1}-(1-t^2)^{-n}\Big)\\
&=& 2n (1-t^2)^{-n-1}- (2n-1) (1-t^2)^{-n}.
\end{eqnarray*}
Therefore
\[
t(1-t^2)^{-n} = 2n \int_0^t (1-s^2)^{-n-1}\mathrm d s-(2n-1)
\int_0^1 (1-s^2)^{-n}\mathrm d s,
\]
which after multiplication by $\omega^{2n}$ and reordering of terms
yields
\[
2n g_n = t+(2n-1) \omega^2 g_{n-1}.
\]
The expression for $n=0$ is straightforward and the increasing
regularity of the sequence of functions follows by induction.
\end{proof}

\begin{lemma}\label{app:lemma1} Let $c_{n,m}$ be the quantities defined in
\eqref{defQnm}. For all $n\ge m\ge 0$
\[
\left(\frac{2n+1}2\right)^{1-2m} \le c_{n,m}^2 \le \left(
\frac{2n+1}2\right)^{1-2m} (m+1)!.
\]
\end{lemma}

\begin{proof} Note that for $m=0$ there is nothing to prove, so we
will assume henceforth that $m\ge 1$. We can write
\begin{equation}\label{app:cnm}
c_{n,m}^2 = \frac{2n+1}2\, \frac1{p_m(n)},
\end{equation}
where
\[
p_m(x):= (x+m)\,(x+m-1)\,\ldots\,(x-m+1)= \prod_{j=-m}^{m-1} (x-j),
\]
is a polynomial of degree $2m$ with zeros in $\{
-m,-m+1,\ldots,m-1\}$. If $x\ge m$, all factors in the definition of
$p_m(x)$ are positive and we can apply the comparison inequality
between the geometric and the arithmetic mean to obtain
\begin{eqnarray*}
p_m(x) & \le & \Big(\frac1{2m} \sum_{j=-m}^{m-1}
(x-j)\Big)^{2m}=\frac1{(2m)^{2m}} \Big(\sum_{j=-m}^{-1}
(x-j)+\sum_{j=1}^m (x-j+1)\Big)^{2m}\\
&=& \frac1{(2m)^{2m}} \Big( m
(2x+1)\Big)^{2m}=\left(\frac{2x+1}{2}\right)^{2m}.
\end{eqnarray*}
If $j\le -1$, then
\[
x-j\ge \frac{2x+1}2, \qquad \forall x,
\]
whereas for $j\ge 0$
\[
x-j \ge \frac1{(j+2)}\left( \frac{2x+1}2\right), \qquad \forall x\ge
j+1.
\]
Therefore for $x\ge m (\ge 1)$
\[
p_m(x) \ge
\left(\frac{2x+1}2\right)^{2m}\prod_{j=0}^{m-1}\frac1{j+2} =
\left(\frac{2x+1}2\right)^{2m} \frac1{(m+1)!}.
\]
We can then apply the two bounds above to obtain
\[
\left(\frac{2n+1}2\right)^{-2m} \le \frac1{p_m(n)} \le
\left(\frac{2n+1}2\right)^{-2m}\, (m+1)!, \qquad \forall n\ge m\ge
1,
\]
which together with the expression of $c_{n,m}$ in \eqref{app:cnm}
proves the result.
\end{proof}

\begin{lemma}\label{app:lemma2}
For all $m,k\ge 0$ and $n\ge m+k$
\[
\frac{(m+1)!}{(m+k+1)!} \left( \frac{2n+1}2\right)^{2k} \le \left(
\frac{c_{n,m}}{c_{n,m+k}}\right)^2 \le \left(
\frac{2n+1}2\right)^{2k}.
\]
\end{lemma}

\begin{proof} For $k=0$ there is nothing to prove, so we assume that
$k\ge 1$. We follow similar steps to those applied when proving
Lemma \ref{app:lemma1}. We first remark that
\[
\left(\frac{c_{n,m}}{c_{n,m+k}}\right)^2 = \prod_{j=m}^{m+k-1}
(n-j)\, \prod_{j=m}^{m+k-1} (n+j+1) =: q_{m,k}(n),
\]
where $q_{m,k}$ is a polynomial of degree $2k$. Since all the
factors are positive when $x \ge m+k$, we can bound
\[
q_{m,k}(x) \le \frac1{(2k)^{2k}} \Big( \sum_{j=m}^{m+k-1} (x-j) +
\sum_{j=m}^{m+k-1} (x+j+1)\Big)^{2k} =\left(
\frac{2x+1}2\right)^{2k}.
\]
Also, for $x\ge m+k$
\[
q_{m,k}(x) \ge \left( \frac{2x+1}2\right)^{2k} \prod_{j =
m}^{m+k-1}\frac1{j+2}= \left(
\frac{2x+1}2\right)^{2k}\frac{(m+1)!}{(m+k+1)!}
\]
and the remaining inequality is proved.
\end{proof}

\begin{lemma}\label{app:lemma3} For all $0\le k\le n$
\begin{equation}\label{app:lemma3eq}
| P_n^{(k)}(t)| \le \frac1{4^k k!}\,(2n+1)^{2k}, \qquad \forall t
\in [-1,1].
\end{equation}
\end{lemma}

\begin{proof} The result is true for $k=0$ since $| P_n(t) | \le 1$
for all $t \in [-1,1]$. We will proceed by induction in $k$.

Elementary computations show that for all $n$
\[
P_n'=(2n-1) P_{n-1}+ P_{n-2}'.
\]
Then we can proceed by induction to prove that
\[
P_n'= \sum_{j=0}^{\lfloor\frac{n-1}2\rfloor} (2n-1-4j) P_{n-1-2j}.
\]
Note that all coefficients are positive for the given values of $j$.
Assume that \eqref{app:lemma3eq} holds for a given $k$. Using the
previous expansion of the first derivative of the Legendre
polynomials we can derive
\[
P_n^{(k+1)} = \sum_{j=0}^{\lfloor\frac{n-1}2\rfloor} (2n-1-4j)
P_{n-1-2j}^{(k)}=\sum_{j=0}^{\lfloor\frac{n-1-k}2\rfloor} (2n-1-4j)
P_{n-1-2j}^{(k)}.
\]
Therefore
\begin{eqnarray*}
|P_n^{(k+1)}(t)| & \le & \frac1{4^k k!}
\sum_{j=0}^{\lfloor\frac{n-1-k}2\rfloor} (2n-1-4j) ( 2
(n-1-2j)+1)^{2k} \\
&=& \frac1{4^k k!} \sum_{j=0}^{\lfloor\frac{n-1-k}2\rfloor}
(2n-1-4j)^{2k+1} \le  \frac1{4^k k!}
\sum_{j=0}^{\lfloor\frac{n-1-k}2\rfloor}
\frac12 \int_{2n-1-4j}^{2n+1-4j} x^{2k+1}\mathrm d x \\
& \le & \frac1{4^k k!} \frac12 \int_0^{2n+1} x^{2k+1} \mathrm d x=
\frac1{4^k k!} \, \frac{(2n+1)^{2k+2}}{2 (2k+2)}
\end{eqnarray*}
and the result is proven for $k+1$.
\end{proof}

\subsection{An abstract lemma}

\begin{lemma}\label{app:lemma4} Let $H, H_1$ and $H_2$ be Hilbert spaces such that
$H_1\subset H$ and $H_2 \subset H$ with continuous injections. If
$\{\xi_n\}$ is an orthonormal basis of $H$ that is complete
orthogonal in $H_1$ and $H_2$ and if
\begin{equation}\label{app:comporth}
c\| \xi_n\|_{H_1} \le \|\xi_n\|_{H_2} \le C\| \xi_n\|_{H_1}, \qquad
\forall n,
\end{equation}
then $H_1=H_2$ and
\[
c\| u\|_{H_1}\le \| u\|_{H_2}\le C\| u\|_{H_1}, \qquad \forall u \in
H_1=H_2.
\]
\end{lemma}

\begin{proof} We first prove that
\[
H_1 = \Big\{ u\in H\,:\, \sum_{n=1}^\infty \| \xi_n\|_{H_1}^2
|(u,\xi_n)_H|^2 < \infty\Big\}
\]
and
\[
\| u\|_{H_1}^2 = \sum_{n=1}^\infty \| \xi_n\|_{H_1}^2
|(u,\xi_n)_H|^2 , \qquad \forall u\in H_1.
\]
Since $\xi_n/\| \xi_n\|_{H_1}$ is an orthonormal basis of $H_1$, we
can decompose
\[
H_1\ni u = \sum_{n=1}^\infty \frac{(u,\xi_n)_{H_1}}{\|
\xi_n\|_{H_1}^2}\xi_n.
\]
The series converges in $H_1$ and therefore in $H$. Hence
\[
(u,\xi_n)_H=\frac{(u,\xi_n)_{H_1}}{\| \xi_n\|_{H_1}^2}
\]
and
\[
\sum_{n=1}^\infty  \| \xi_n\|_{H_1}^2 |(u,\xi_n)_H|^2 =
\sum_{n=1}^\infty \frac{|(u,\xi_n)_{H_1}|^2}{\|\xi_n\|_{H_1}^2} = \|
u\|_{H_1}^2.
\]
A similar characterization of $H_2$ and comparison of the weights by
means of \eqref{app:comporth} shows that $H_1=H_2$. The inequality
for the norms is a consequence of the series form of the norms in
$H_1$ and $H_2$.
\end{proof}

\section{The case of Legendre polynomials}

We can start the definition of the Hilbert scale defined by Legendre
polynomials in the following way. In the space
\[
D_0:= \{ u \in L^2(-1,1)\,:\, \omega\,u'\in L^2(-1,1)\}
\]
we consider its natural inner product
\[
(u,v)_{D_0} :=\int_{-1}^1 \omega^2 u'v'+\int_{-1}^1 u\,v.
\]
Because in this case $\lambda=0$ is an eigenvalue of the
differential operator $\mathcal L_0u:=-(\omega^2 u')'$, to simplify
the exposition we will simply translate the spectrum by adding an
identity operator and we will study instead $\widetilde{\mathcal
L}_0u:=\mathcal L_0 u+u$ as in \cite{BrLiTuWe:2009} and related
references. Note that $H^1(-1,1)\subset D_0\subset
H^1_{\mathrm{loc}}(-1,1)$ with continuous injection.

\begin{proposition}\label{PdenseD0} $\mathbb P$ is a dense subset of $D_0$.
\end{proposition}

\begin{proof}
Let us start by proving the following assertion: {\em if $u \in
D_0$, then $\omega^2u \in H^1_0(-1,1)$.} The fact that $\omega^2u\in
H^1(-1,1)$ is straightforward to prove using Leibniz's rule and the
definition of $D_0$. Hence $\omega^2 u\in \mathcal C[-1,1]$ and
since $u=\omega^{-2} \,(\omega^2\, u)\in L^2(-1,1)$, necessarily
$\omega^2 u \in \mathcal C_0$.

It is clear that $\mathbb P\subset D_0$. If
\[
\int_{-1}^1 \omega^2 u'p'+\int_{-1}^1 u\, p=0, \qquad \forall p \in
\mathbb P,
\]
we can apply the integration by parts lemma (Lemma \ref{intbyparts})
and easily show that
\[
\int_{-1}^1 u (-(\omega^2 p')'+p)=\int_{-1}^1
u\,(\widetilde{\mathcal L}_0p) =0, \qquad \forall p \in \mathbb P.
\]
It suffices to take $p=P_n$ for all values of $n$ and use that
Legendre polynomials form a Hilbert basis of $L^2(-1,1)$ to prove
that $u=0$. This proves that $\mathbb P$ is dense in $D_0$.
\end{proof}

Thanks to Proposition \ref{PdenseD0} we know that $\{ Q^0_n\,:\,
n\ge 0 \}$ is a complete orthogonal set in $D_0$ and that
\[
\| u\|_{D_0}^2 = \sum_{n=0}^\infty \big( (n(n+1)+1\big)
|(u,Q^0_n)_{L^2(-1,1)}|^2 , \qquad \forall u \in D_0.
\]
The corresponding Green's operator
\[
G_0u:= \sum_{n=0}^\infty\frac1{n(n+1)+1} (u,Q_n^0)_{L^2(-1,1)} \,
Q_n^0
\]
(note the translation of eigenvalues with respect to the operator
defined in \eqref{defGm2} due to the addition of the identity
operator to $\mathcal L_0$ to translate the zero eigenvalue) is
equivalent to the operator defined by $G_0 f=u$, where $u$ is the
unique solution of
\begin{equation}\label{defG0}
u\in D_0 , \qquad (u,v)_{D_0}=(f,v)_{L^2(-1,1)}, \qquad \forall v\in
D_0.
\end{equation}
With help of this representation we will be able to describe the
range $\mathcal R(G_0)$ as a Sobolev type space in two different
ways. Note that unlike in the cases $m\ge 1$, the range of $\mathcal
R(G_0)$ contains some generalized boundary conditions.

\begin{proposition}
\begin{eqnarray}\label{RG0a}
\mathcal R(G_0) &=& \{ u \in L^2(-1,1)\,:\, \omega^2 u'\in
H^1_0(-1,1)\}\\&=& \{ u\in D_0\,:\, \mathcal L_0u\in L^2(-1,1), \,
(\omega^2u')(-1)=(\omega^2u')(1)=0\}. \label{RG0b}
\end{eqnarray}
\end{proposition}

\begin{proof} Let $A_1$ and $A_2$ be the respective sets in the
right--hand side of \eqref{RG0a} and \eqref{RG0b}. Note that in
$A_1$ the condition $\omega\,u'\in L^2(-1,1)$ (that appears in the
definition of $D_0$) has been relaxed to $\omega^2 u'\in L^2(-1,1)$.
Therefore $A_2\subset A_1$.

We next prove that $\mathcal R(G_0) \subset A_2$. If $u=G_0 f$ with
$f \in L^2(-1,1)$, using \eqref{defG0} with a general $\psi\in
\mathcal D(-1,1)\subset D_0$ we prove that $\widetilde{\mathcal
L_0}u=f$ in the sense of distributions. Note that this implies that
$\mathcal L_0 u=f-u\in L^2(-1,1)$. Substituting this differential
equation in \eqref{defG0} we obtain that
\[
\int_{-1}^1 (\omega^2 u')v' +\int_{-1}^1 (\omega^2 u')'v =0, \qquad
\forall v\in D_0.
\]
If we consider the function $\tilde u:=\omega^2 u'\in H^1(-1,1)$,
then the equality
\[
\int_{-1}^1 ( \tilde u v'+ \tilde u'v)=0, \qquad \forall v \in
H^1(-1,1)
\]
(note that $H^1(-1,1)\subset D_0$) is a weak form of the boundary
conditions $\tilde u(-1)=\tilde u(1)=0$. This completes the proof
that $\mathcal R(G_0) \subset A_2$.

We finally have to prove that $A_1 \subset \mathcal R(G_0)$. Let now
$u\in A_1$. Because $\omega^2u'\in H^1_0(-1,1)$, then by Lemma
\ref{lemma:H10}(a), $\omega\,u' \in \mathcal C_0\subset L^2(-1,1)$,
so $u\in  D_0$. Also
\[
\int_{-1}^1 \omega^2 u' v' +\int_{-1}^1 (\omega^2 u')' v=0, \qquad
\forall v\in H^1(-1,1).
\]
If we define $f:=\widetilde{\mathcal L}_0 u \in L^2(-1,1)$, it
follows readily that
\[
(u,v)_{D_0}=(f,v)_{L^2(-1,1)}, \qquad \forall v \in \mathbb P
\]
and by density we show that $u=G_0 f$, which completes the proof.
\end{proof}

\begin{proposition} If $u\in \mathcal R(G_0)$, then $u'\in
L^2(-1,1)$.
\end{proposition}

\begin{proof} By Theorem \ref{app:the2} (let us emphasize again that
Section \ref{oldsection4} is independent of the two sections that
precede it), we have a different characterization of $\mathcal
R(G_0)=H^2_0$ as $X^2_0$. If $u\in\mathcal R(G_0)$, then
\[
2t\,u'=\omega^2u''-(\omega^2u')'\in L^2(-1,1).
\]
On the other hand $u'\in L^2_{\mathrm{loc}}(-1,1)$, so we can divide
by $t$ and ensure that $u'\in L^2(-1,1)$.\end{proof}

This last result appears in \cite{ArLiMa:2002}, quoted as already
been proved in the unpublished preprint \cite{EvMa:1988}.

\end{document}